\documentclass{amsart}
\usepackage{tikz}
\usepackage{xcolor}
\usepackage{amssymb,latexsym,amsmath,extarrows}
\usepackage{graphicx,mathrsfs,comment}
\usepackage{hyperref,url}
\usepackage{bbm}   
\usepackage{scalerel,stackengine}
\stackMath
\newcommand\widecheck[1]{%
\savestack{\tmpbox}{\stretchto{%
  \scaleto{%
    \scalerel*[\widthof{\ensuremath{#1}}]{\kern-.6pt\bigwedge\kern-.6pt}%
    {\rule[-\textheight/2]{1ex}{\textheight}}
  }{\textheight}%
}{0.5ex}}%
\stackon[1pt]{#1}{\scalebox{-1}{\tmpbox}}%
}
\usepackage{ esint }

\usepackage{esint}

\makeatletter
\newcommand\@avprod[2]{%
  {\sbox0{$\m@th#1\prod$}%
   \vphantom{\usebox0}%
   \ooalign{%
     \hidewidth
     \smash{\vrule height\dimexpr\ht0+1pt\relax depth\dimexpr\dp0+1pt\relax}%
     \hidewidth\cr
     $\m@th#1\prod$\cr
   }%
  }%
}
\newcommand{\avprod}{\mathop{\mathpalette\@avprod\relax}\displaylimits}

\usepackage{amstext}
\usepackage{bbm}

\numberwithin{equation}{section}

\newtheorem{theorem}{Theorem}[section]
\newtheorem{lemma}[theorem]{Lemma}

\newtheorem{proposition}[theorem]{Proposition}

\newtheorem{definition}[theorem]{Definition}
\newtheorem{corollary}[theorem]{Corollary}

\newcommand{\e}{\epsilon}

\newcommand{\mP}{\mathcal{P}}
\newcommand{\mPb}{\mathcal{P}_{R^{-\beta}} }
\newcommand{\dpm}{D_{p,\mathrm{mul}}}

\newcommand{\R}{\mathbb{R}}

\newcommand{\supp}{\textup{supp}}

\begin{document}

\title[]{On the small cap decoupling for the moment curve in $\mathbb{R}^3$}

\author{Dominique Maldague} \address{Dominique Maldague\\  Centre for Mathematical Sciences, Wilberforce Road, Cambridge CB3 0WA, UK}\email{dm672@cam.ac.uk}

\author{Changkeun Oh}\address{ Changkeun Oh\\ Department of Mathematical Sciences and RIM, Seoul National University, Republic of Korea} \email{changkeun@snu.ac.kr}


\maketitle

\begin{abstract} This paper proves sharp small cap decoupling estimates for the moment curve $\mathcal{M}^n=
\{(t,t^2,\ldots,t^n):0\le t\le 1\}$ in the remaining small cap parameter ranges for $\R^2$ and $\R^3$. 
\end{abstract}



\section{Introduction}

Small cap decoupling is an $L^p$ estimate of exponential sums with frequencies lying in curved sets and satisfying a spacing condition (one frequency per small cap). The general setup was introduced in \cite{MR4153908}, where they proved sharp estimates for frequencies on curves in $\R^2$ and certain other special cases. In \cite{MR4153908}, the authors recorded a conjecture (Conjecture 2.5) about sharp small cap decoupling estimates for frequencies lying on the moment curve $\mathcal{M}^n=\{(t,t^2,\ldots,t^n):0\le t\le 1\}$ and they made partial progress on the $n=3$ case. In \cite{guth2022small}, Guth and Maldague proved a more general small cap decoupling result which fully resolved the $n=3$ case of Conjecture 2.5 from \cite{MR4153908}. The main result of this paper is sharp small cap decoupling estimates for $\mathcal{M}^3$ in the full range of parameters not addressed by Conjecture 2.5 in \cite{MR4153908} or in \cite{guth2022small}. In particular, we have the following. 

\begin{theorem}\label{23.12.31.thm11} For $\sigma \ge 0$ and $p\ge 2$, we have
\begin{equation*}
    \Big( \int_{[0,1]^2 \times [0,\frac{1}{N^{\sigma}}]} \Big| \sum_{k=1}^{N} a_k e\big(x \cdot (k,k^2,k^3)\big) \Big|^p\,dx \Big)^{\frac1p} \leq C_{\epsilon}N^{\epsilon} D_p(\sigma,N)(\sum_{k=1}^{N}|a_{k}|^p)^{\frac1p}
\end{equation*}
for any $a_k \in \mathbb{C}$, in which 
\[ D_p(\sigma,N)=  \begin{cases}
N^{\frac12-\frac{\sigma+1}{p} }+N^{1-\frac{7}{p}}     &\qquad\text{if}\qquad 0\leq \sigma < 5/2 \\
    N^{\frac12-\frac{\sigma+1}{p} }+N^{\frac{\sigma}{3}-\frac{7\sigma}{3p} }+N^{1-\frac{7}{p}} &\qquad\text{if}\qquad 5/2\leq \sigma < 3 \\
    N^{\frac12-\frac{\sigma}{p} }+N^{1-\frac{3+\sigma}{p}}&\qquad\text{if}\qquad 3\leq \sigma  
\end{cases}. \]

\end{theorem}
Theorem \ref{23.12.31.thm11} is sharp, up the $C_\e N^\e$ factor. We demonstrate this in \textsection\ref{sharpness}. Note that for $5/2 \leq \sigma \leq 3$, we have
\begin{equation*}
    N^{\frac12-\frac{\sigma}{p} }+N^{\frac{\sigma}{3}+(1-\frac{7\sigma}{3})\frac1p }+N^{1-\frac{6}{p}} \sim
    \begin{cases}
    N^{\frac12-\frac{\sigma}{p} } \; \mathrm{ for } \;p< \frac{8\sigma-6}{2\sigma-3}  \\  N^{\frac{\sigma}{3}+(1-\frac{7\sigma}{3})\frac1p } \; \mathrm{\; for\; } \frac{8\sigma-6}{2\sigma-3} <p <7 \\ N^{1-\frac{6}{p}} \; \mathrm{\; for\; } 7<p
    \end{cases}
\end{equation*}
Conjecture 2.5 of \cite{MR4153908} is a special case of Theorem \ref{23.12.31.thm11} for the range $0\le \sigma\le 2$. Bombieri-Iwaniec \cite{MR881101} and Bourgain \cite{MR3736494} resolved a special case of Theorem \ref{23.12.31.thm11} for $\sigma=2$. Guth, Demeter, and Wang verified a special case of the analog of Theorem \ref{23.12.31.thm11} in the range $0\le \sigma\le \frac{3}{2}$. Guth and Maldague proved a more general version of Theorem \ref{23.12.31.thm11} in the range $0\le \sigma\le 2$, which covered the range of the parameter $\sigma$ addressed by Conjecture 2.5 in \cite{MR4153908}.

If $\sigma\ge 3$, then the function $\sum_{k=1}^N a_ke(x\cdot(k,k^2,k^3))$ is very slowly varying in the third coordinate. This means that $L^p([0,1]^2\times[0,N^{-\sigma}])$ estimates for $\sigma\ge 3$ are trivially controlled using Fubini's theorem by the sharp estimates for $2$-dimensional exponential sums with frequencies on the parabola \cite{MR3374964}. We therefore omit the proof of Theorem \ref{23.12.31.thm11} in the range $\sigma\ge 3$.

Let us state a continuous setup for small cap decoupling. For convenience, we use the same notation in \cite{guth2022small}. For $\beta \in [\frac13,\infty)$, define $\mathcal{M}^3(R^{\beta},R)$ by
\begin{equation}
    \{(\xi_1,\xi_2,\xi_3): \xi_1 \in [0,1], |\xi_2-\xi_1^2| \leq R^{-2\beta}, |\xi_3-3\xi_1\xi_2+2\xi_1^3| \leq R^{-1} \}.
\end{equation}
We will partition this set into small caps of the form
\begin{equation}\label{07.29.12}
\begin{split}
    \gamma:=\{(\xi_1,\xi_2,\xi_3) : l R^{-\beta} \leq &\xi_1 <(l+1)R^{-\beta}, 
    \\&|\xi_2-\xi_1^2| \leq R^{-2\beta}, |\xi_3-3\xi_1\xi_2+2\xi_1^3| \leq R^{-1} \}
\end{split}
\end{equation}
for some integer $l$, $0\leq l <R^{\beta}$. Denote by $\mathcal{P}_{R^{-\beta}}$ the collection of $\gamma$'s. We are ready to state the small cap decoupling theorem.
\begin{theorem}\label{23.12.31.thm12} For $1<\beta \leq 2$, $p\geq 2$ and $\epsilon>0$, we have
    \begin{equation}\label{24.01.14.14}
        \|f\|_{L^p} \leq C_{\epsilon} R^{\epsilon} \big( R^{\beta(\frac12-\frac1p)}+R^{\beta(1-\frac{4}{p})-\frac1p} \big) \big(\sum_{\gamma \in \mPb }\|f_{\gamma}\|_{L^p}^p \big)^{\frac1p}
    \end{equation}
    for any $f$ with Fourier transform supported in $\mathcal{M}^3(R^{\beta},R)$. 
    
    For $\beta >2$, $p \geq 2$, and $\epsilon>0$,
    \begin{equation}
        \|f\|_{L^p} \leq C_{\epsilon} R^{\epsilon} \big( R^{\beta(\frac12-\frac1p)}+R^{(\beta-\frac13)(1-\frac4p)}+R^{\beta(1-\frac{4}{p})-\frac1p} \big) \big(\sum_{\gamma \in \mPb }\|f_{\gamma}\|_{L^p}^p \big)^{\frac1p}
    \end{equation}
    for any $f$ with Fourier transform supported in $\mathcal{M}^3(R^{\beta},R)$.
\end{theorem}

The exponents of $R$ of \eqref{24.01.14.14} coincide when $p=6+2/\beta$.
For $\beta \geq 2$,
\begin{equation*}
    \max\{R^{\beta(\frac12-\frac1p)},R^{\beta(1-\frac{4}{p})-\frac1p}, R^{(\beta-\frac13)(1-\frac4p)} \}=
    \begin{cases}
    R^{\beta(\frac12-\frac1p)} \; \mathrm{ for } \;p<6+\frac{2}{\beta}  \\  R^{(\beta-\frac13)(1-\frac4p)} \; \mathrm{\; for\; } 6+\frac{2}{\beta} <p <7 \\ R^{\beta(1-\frac{4}{p})-\frac1p} \; \mathrm{\; for\; } 7<p
    \end{cases}
\end{equation*}

Finally, since this paper is about small cap decoupling in the remaining parameter ranges, we record the following proposition for the parabola. 

\begin{proposition}\label{parabolathm}
    For $1 \leq \sigma \leq 2$ and $2\le p <\infty$, we have
\begin{equation*}
    \Big( \int_{[0,1] \times [0,\frac{1}{N^{\sigma}}]} \Big| \sum_{k=1}^{N} a_k e\big(x \cdot (k,k^2)\big) \Big|^p\,dx \Big)^{\frac1p} \leq C_{\epsilon}N^{\epsilon}\big(N^{\frac{\sigma}{2}(1-\frac{4}{p})}+N^{1-\frac{4}{p}} \big) (\sum_{k=1}^{N}|a_{k}|^p)^{\frac1p}
\end{equation*}
    for any $a_k \in \mathbb{C}$. 
\end{proposition}

Note that for the case $1 \leq \sigma <2$ the critical exponent is $p=4$.
 
\subsection{Sharpness of theorems\label{sharpness}}
Let us show the sharpness of Theorem \ref{23.12.31.thm11}. We will take three examples to show the sharpness. Let us begin with the example that $a_{k}=1$ for all $k$. We will show that
\begin{equation*}
    N^{1-\frac{\max(3+\sigma,6)}{p} }   \lesssim \Big( \int_{[0,1]^2 \times [0,\frac{1}{N^{\sigma}}]} \Big| \sum_{k=1}^{N}  e\big(x \cdot (k,k^2,k^3)\big) \Big|^p\,dx \Big)^{\frac1p}.
\end{equation*}
By evaluating the exponential sum on the smaller set $[0,N^{-1}] \times [0,N^{-2}] \times [0,N^{-\max(3,\sigma)}]$, we get the bound. The second example is a random example. We take a random sequence $a_{k} \in \{-1,1\}$, and this gives the bound
\begin{equation*}
    N^{\frac12-\frac{\sigma}{p} }   \lesssim \Big( \int_{[0,1]^2 \times [0,\frac{1}{N^{\sigma}}]} \Big| \sum_{k=1}^{N}  e\big(x \cdot (k,k^2,k^3)\big) \Big|^p\,dx \Big)^{\frac1p}.
\end{equation*}
The last example is a ``lower dimensional'' example. Take
\begin{equation}
    a_k= 
    \begin{cases}
            1 \text{ for } 0 \leq k \leq N^{\sigma/3} \\
            0 \text{ otherwise }
                \end{cases}
\end{equation}
Then we have
\begin{equation}
\begin{split}
    \mathrm{LHS} &\gtrsim N^{-\frac{\sigma}{p}}
    \Big( \int_{[0,1]^2 } \Big| \sum_{k=1}^{N^{\sigma/3}} e\big(x \cdot (k,k^2)\big) \Big|^p\,dx \Big)^{\frac1p}
    \\& \gtrsim N^{-\frac{\sigma}{p}}
    \Big( \int_{[0,N^{-\sigma/3}]\times [0,N^{-2\sigma/3}] } \Big| \sum_{k=1}^{N^{\sigma/3}} e\big(x \cdot (k,k^2)\big) \Big|^p\,dx \Big)^{\frac1p}
    \\&\gtrsim N^{-\frac{(\frac13+\frac23+1)\sigma}{p}}N^{\frac{\sigma}{3}} \sim N^{-\frac{2\sigma}{p}}N^{\frac{\sigma}{3}}.
\end{split}
\end{equation}

The sharpness of Proposition \ref{parabolathm} follows from analyzing the function $f(x)=\sum_{k=1}^N e(x\cdot(k,k^2))$. The integrand satisfies $|f(x)|^p\gtrsim N$ on a box of dimensions $N^{-1}\times N^{-2}$ centered at the origin, so 
\[ \text{LHS}\gtrsim N\cdot N^{-\frac{3}{p}}\sim N^{1-\frac{4}{p}}(\sum_{k=1}^N|1|^p)^{\frac{1}{p}}. \]
The other contribution to $\text{LHS}$ comes from the $\mathcal{R}=
[-N^{1-\sigma},N^{1-\sigma}]\times[-N^{-\sigma},N^{-\sigma}]$. Indeed, let $\eta$ be a Schwartz function $\eta:\R^2\to[0,\infty)$ satisfying $|\eta|\sim 1$ on $\mathcal{R}$, $\widecheck{\eta}$ is real-valued with $|\widecheck{\eta}|\sim N^{1-2\sigma}$ on $\mathcal{R}^*=[-N^{\sigma-1},N^{\sigma-1}]\times[-N^{\sigma},N^{\sigma}]$, and $\supp \widecheck\eta\subset2\mathcal{R}^*$.  

Using $L^2$ orthogonality, we have
\begin{align*} 
\|\sum_{k=1}^Ne(x\cdot(k,k^2))&\|_{L^2(\mathcal{R})}^2\sim \int|\sum_{k=1}^Ne(x\cdot(k,k^2))|^2\eta(x)dx\\
    &\sim \sum_{k,k'} \widecheck{\eta}(k-k',k^2-(k')^2) \sim N\cdot N^{\sigma-1}N^{1-2\sigma}. 
\end{align*}
Then since $2\le p$, H\"{o}lder's inequality gives, 
\[  \|\sum_{k=1}^Ne(x\cdot(k,k^2))\|_{L^2(\mathcal{R})}\le |\mathcal{R}|^{\frac{1}{2}-\frac{1}{p}}\|\sum_{k=1}^Ne(x\cdot(k,k^2))\|_{L^p(\mathcal{R})} . \]
Combining the previous two inequalities yields 
\[ \text{LHS}\gtrsim |\mathcal{R}|^{\frac{1}{p}-\frac{1}{2}}N^{\frac{1-\sigma}{2}}
\sim N^{\frac{\sigma}{2}(1-\frac{4}{p})}(\sum_{k=1}^N|1|^p)^{\frac{1}{p}}, \]
which concludes the discussion of sharpness for Proposition \ref{parabolathm}. 

\subsection{Notations}

For $A,B,C \in \mathbb{R} \cap (0,\infty)$, $B_{A,B,C}(x)$ is defined by
\begin{equation}
    B_{A,B,C}(x):=x+\big([0,A] \times [0,B] \times [0,C]\big) \subset \R^3.
\end{equation}
When $x=(0,0,0)$, we sometimes use $B_{A,B,C}$ for $B_{A,B,C}(x)$. When $A=B=C$, we sometimes use $B_{A}^{(3)}$ or $B_{A}$ for $B_{A,B,C}$. Similarly, $B_{A,B}(x')$ is defined by
\begin{equation}
    B_{A,B}(x'):=x'+\big([0,A] \times [0,B] \big) \subset \R^2.
\end{equation}
When $x'=(0,0)$, we sometimes use $B_{A,B}$ for $B_{A,B}(x')$. When $A=B$, we sometimes use $B_{A}^{(2)}$ for $B_{A,B}$.
For a ball $B_R(a) \subset \R^3$, define the weight function by
\begin{equation}
    w_{B_R(a)}(x):=\big(1+\frac{|x-a|}{R} \big)^{-1000},
\end{equation}
and the $L^1$ normalized weight function by
\begin{equation}
    w_{B_R(a)}^{\#}(x):=\frac{1}{R^3}\big(1+\frac{|x-a|}{R} \big)^{-1000}.
\end{equation}
The weight functions $w_{B_R^{(2)}(a)}^{\#}(x_1,x_2)$ and $w_{B_R^{(1)}(a)}^{\#}(x_1)$ are defined similarly. We define an weighted integral 
\begin{equation}\label{24.11.11.111}
    \|f\|_{L^1(w^{\#}_{B_R}(a))}:= \int_{\mathbb{R}^3} |f(x)| w_{B_R(a)}^{\#}(x) \,dx.
\end{equation}
Define $\|f\|_{L^1(w_{B_R^{(2)}(a)}^{\#} \times w_{B_R^{(1)}(a)}^{\#} )}$ similarly. For a measurable set $A$,
\begin{equation}
    \fint_{A}|f|:=\frac{1}{|A|}\int |f|, \;\;\; 
\end{equation} 
and define
\begin{equation}
    \|f\|_{L^p_{\#}(B_R)}:= \Big(  \fint_{B_R}|f|^p \Big)^{\frac1p}.
\end{equation}
We use the notation
\begin{equation}
    \mathrm{RapDec}(R)=R^{-1000}.
\end{equation}
For a sequence $\{ a_i\}_{i=1}^{j}$ we introduce a notation for the geometric mean
\begin{equation}
\avprod_{i=1}^{j}a_i := \prod_{i=1}^j |a_i|^{\frac1j}.
\end{equation}

\subsection{Acknowledgements}
 CO was supported by the NSF grant DMS-1800274 and POSCO Science Fellowship of POSCO TJ Park Foundation.

\section{Reduction to a multilinear decoupling inequality}

We first reduce a global decoupling inequality to a local decoupling inequality. This argument is standard (for example, see \cite[Lemma 9]{guth2022small}), so we omit the proof.
\begin{proposition}
    Theorem \ref{23.12.31.thm12} follows from the following: 
    
    For $1 \leq \beta \leq 2$, $p\geq 2$, and $\epsilon>0$,
    \begin{equation}
        \|f\|_{L^p(B_{R^{2\beta},R^{2\beta},R})} \leq C_{\epsilon} R^{\epsilon} \big( R^{\beta(\frac12-\frac1p)}+R^{\beta(1-\frac{4}{p})-\frac1p} \big) \big(\sum_{\gamma \in \mP_{R^{-\beta}} }\|f_{\gamma}\|_{L^p(\mathbb{R}^3)}^p \big)^{\frac1p}
    \end{equation}
    for any $f$ with Fourier transform supported in $\mathcal{M}^3(R^{\beta},R)$.

    For $\beta \geq 2$, $p \geq 2$, and $\epsilon>0$,
    \begin{equation*}
    \begin{split}
&\|f\|_{L^p(B_{R^{2\beta},R^{2\beta},R})} 
\\&\leq C_{\epsilon} R^{\epsilon}
        \max\{R^{\beta(\frac12-\frac1p)},R^{\beta(1-\frac{4}{p})-\frac1p}, R^{(\beta-\frac13)(1-\frac4p)} \}
        \big(\sum_{\gamma \in \mP_{R^{-\beta}} }\|f_{\gamma}\|_{L^p(\mathbb{R}^3)}^p \big)^{\frac1p}
        \end{split}
    \end{equation*}
    for any $f$ with Fourier transform supported in $\mathcal{M}^3(R^{\beta},R)$.
\end{proposition}

Let $D_p(R)$ be the smallest constant such that 
\begin{equation}
        \|f\|_{L^p(B_{R^{2\beta},R^{2\beta},R})} \leq D_p(R) \big(\sum_{\gamma \in \mP_{R^{-\beta}} }\|f_{\gamma}\|_{L^p(\mathbb{R}^3)}^p \big)^{\frac1p}
    \end{equation}
    for any $f$ with Fourier transform supported in $\mathcal{M}^3(R^{\beta},R)$.
To state the multilinear decoupling constant, let us introduce transversality.
We say that three intervals $I_1,I_2,I_3 \subset [0,1]$ are $C$-transverse if
\begin{equation}
    \min_{i,j: i \neq j}\mathrm{dist}(I_i,I_j) > C>0.
\end{equation}
For $K \geq 1$, let $D_{p,\mathrm{mul}}(R,K^{-1})$ be the smallest constant such that for any $K^{-1}$-transverse $I_i$ it holds that 
\begin{equation}
        \Big\|\avprod_{i=1}^{3}f_i\Big\|_{L^p(B_{R^{2\beta},R^{2\beta},R})} \leq \dpm(R,K^{-1}) \big(\sum_{i=1}^{3}\sum_{\gamma \in \mP_{R^{-\beta}} }\|(f_i)_{\gamma}\|_{L^p(\mathbb{R}^3)}^p \big)^{\frac1p}
    \end{equation}
for any $f_i$ which is a Fourier restriction of $f$ to $ I_i \times \mathbb{R}^2$ and $f$ is a function whose Fourier support is in $\mathcal{M}^3(R^{\beta},R)$.

\begin{theorem}\label{24.01.20.thm22}
Let $\beta \geq 1$.
For given $\epsilon>0$, there exists $K \geq 1$ such that
\begin{equation}
\begin{split}
    D_p(R) \leq &C_{\epsilon}R^{\epsilon} 
\sup_{1 \leq R' \leq R} D_{p,\mathrm{mul}}(R',K^{-1})
    \\&+ C_{\epsilon}R^{\epsilon}\max\{R^{\beta(\frac12-\frac1p)},R^{\beta(1-\frac{4}{p})-\frac1p}, R^{(\beta-\frac13)(1-\frac4p)} \}.
\end{split}
\end{equation}
for all $R \geq 1$.  In particular, for $1 \leq \beta \leq 2$, we have
\begin{equation}
\begin{split}
    D_p(R) \leq &C_{\epsilon}R^{\epsilon} \sup_{1 \leq R' \leq R}D_{p,\mathrm{mul}}(R',K^{-1})
    \\&+ C_{\epsilon}R^{\epsilon}\max\{R^{\beta(\frac12-\frac1p)},R^{\beta(1-\frac{4}{p})-\frac1p} \}.
\end{split}
\end{equation}
\end{theorem}

\begin{proof}
    A standard broad-narrow analysis gives
    \begin{equation}
    \begin{split}
        \|f\|_{L^p(B_{R^{2\beta},R^{2\beta},R})} &\lesssim K^{100} \Big\|\avprod_{i=1}^{3}f_{\tau_i}\Big\|_{L^p(B_{R^{2\beta},R^{2\beta},R})} 
        \\&+ \big(\sum_{\tau \in \mP_{K^{-1}} }\|f_{\tau}\|_{L^p(B_{R^{2\beta},R^{2\beta},R})}^p \big)^{\frac1p}.
    \end{split}
    \end{equation}
    Let us fix $\tau \in \mathcal{P}_{K^{-1}}$. For convenience, assume that $\tau$ contains the origin.
    We apply an affine transformation $T$ mapping $\tau$ to $[0,1]^3$. Define $\hat{g}(\xi):=\hat{f}(T^{-1}\xi)$. After this change of variables, we have
    \begin{equation}
        \|f_{\tau}\|_{L^p(B_{R^{2\beta},R^{2\beta},R})} \sim \|g\|_{L^p(B_{R^{2\beta}/K,R^{2\beta}/K^2,R/K^3})}.
    \end{equation}
    Note that
    \begin{equation}
        (R/K^3)^{2\beta} \leq R^{2\beta}/K^2.
    \end{equation}
    We apply the definition of $D_p(R)$ and obtain
    \begin{equation}
        \|g\|_{L^p(B_{R^{2\beta}/K,R^{2\beta}/K^2,R/K^3})}
        \leq D_p(R/K^3)
        \big(\sum_{\gamma' \in \mP_{K^{3\beta}R^{-\beta}} } \|g_{\gamma'}\|_{L^p}^p \big)^{\frac1p}.
    \end{equation}
    After changing back to the original variables, we obtain
    \begin{equation}
        \|f_{\tau}\|_{L^p} \lesssim D_p(R/K^3)
        \big(\sum_{\gamma'' \in \mP_{K^{3\beta-1}R^{-\beta}} } \|f_{\gamma''}\|_{L^p}^p \big)^{\frac1p}.
    \end{equation}
    By the decoupling for the parabola (\cite{MR3374964}),
    \begin{equation*}
        \|f_{\gamma''}\|_{L^p} \lesssim_{\epsilon} K^{\epsilon} \max\{(K^{3\beta-1})^{1-\frac{4} {p}}, (K^{3\beta-1})^{\frac12-\frac1p} \} \big(\sum_{\gamma \in \mPb: \gamma \subset \gamma'' }\|f_{\gamma}\|_{L^p}^p \big)^{\frac1p}.
    \end{equation*}
    To summarize,
    \begin{equation*}
    \begin{split}
        \|f\|_{L^p(B_{R^{2\beta},R^{2\beta},R})} &\lesssim_{\epsilon} K^{100} \Big\|\avprod_{i=1}^{3}f_{\tau_i}\Big\|_{L^p(B_{R^{2\beta},R^{2\beta},R})} 
        \\&+ K^{\epsilon}D_p(RK^{-3})\max\{(K^{3\beta-1})^{1-\frac{4} {p}}, (K^{3\beta-1})^{\frac12-\frac1p} \} \big(\sum_{\gamma \in \mP_{R^{-\beta}} }\|f_{\gamma}\|_{L^p}^p \big)^{\frac1p}.
    \end{split}
    \end{equation*}
    By the definition of $D_p(R)$, we have
    \begin{equation}
        D_p(R) \lesssim_{\epsilon} 
        D_{p,\mathrm{mul}}(R,K^{-1})+K^{\epsilon}
        D_p(RK^{-3})\max\{(K^{3\beta-1})^{1-\frac{4} {p}}, (K^{3\beta-1})^{\frac12-\frac1p}\}.
    \end{equation}
    We apply this inequality repeatedly, and obtain
    \begin{equation}
        D_p(R) \lesssim_{\epsilon} 
       \sup_{1 \leq R_1 \leq R} D_{p,\mathrm{mul}}(R_1,K^{-1})+R^{\epsilon}
       \max\{(R^{\beta-\frac13})^{1-\frac{4} {p}}, (R^{\beta-\frac13})^{\frac12-\frac1p}\}.
    \end{equation}
    This completes the proof.
\end{proof}

Let $I_1,I_2,I_3$ be $C$-transverse. Define
\begin{equation}
    U_{\alpha,\mathrm{mul} }:=\big\{x \in B_{R^{2\beta},R^{2\beta},R}: |f_1(x)f_2(x)f_3(x)|^{\frac13} \sim \alpha \big\}.
\end{equation}
We will prove a level set estimate.

\begin{proposition}\label{24.01.18.prop23}
For $p \geq 2$ and for any $\alpha >0$, we have
    \begin{equation*}
        \alpha^p|U_{\alpha,\mathrm{mul} }| \leq C_{\epsilon} \max \{R^{p \beta(\frac12-\frac1p)},R^{p\beta(1-\frac4p)-1} \} R^{\epsilon} \sum_{\gamma} \|f_{\gamma}\|_{L^4}^4 \big( \sup_{\gamma} \|f_{\gamma}\|_{L^{\infty}} \big)^{{p-4}}.
    \end{equation*}
\end{proposition}

\begin{proof}[Proof that Proposition \ref{24.01.18.prop23} implies Theorem \ref{23.12.31.thm12}]
For $p \geq 2$, by a standard pigeonholing argument (for example, see \cite[Proposition 4]{guth2022small}), Proposition \ref{24.01.18.prop23} implies
    \begin{equation}
        \alpha^p|U_{\alpha,\mathrm{mul} }| \leq C_{\epsilon} \max \{R^{p \beta(\frac12-\frac1p)},R^{p\beta(1-\frac4p)-1} \} R^{\epsilon} \sum_{\gamma} \|f_{\gamma}\|_{L^p}^p.
    \end{equation}    Since this is true for any $\alpha>0$, this implies
    \begin{equation*}
        \Big\|\avprod_{i=1}^{3}f_i\Big\|_{L^p(B_{R^{2\beta},R^{2\beta},R})}^p \leq  
        C_{\epsilon} \max \{R^{p \beta(\frac12-\frac1p)},R^{p\beta(1-\frac4p)-1} \} R^{\epsilon}
        \sum_{\gamma \in \mP_{R^{-\beta}} }\|f_{\gamma}\|_{L^p}^p. 
    \end{equation*}
    By the definition of $D_{p,\mathrm{mul}}(R,K^{-1})$, this says that
    \begin{equation}
        D_{p,\mathrm{mul}}(R,K^{-1}) \leq C_{\epsilon} \max \{R^{ \beta(\frac12-\frac1p)},R^{\beta(1-\frac4p)-\frac1p} \} R^{\epsilon}.
    \end{equation}
    Theorem \ref{24.01.20.thm22} gives
    \begin{equation}
        D_p(R) \leq C_{\epsilon}R^{\epsilon}
       \max\{R^{\beta(\frac12-\frac1p)},R^{\beta(1-\frac{4}{p})-\frac1p}, R^{(\beta-\frac13)(1-\frac4p)} \}.
    \end{equation}
    This completes the proof.
\end{proof}

\section{Preliminary}

Let us introduce some notations. Let $\beta \geq 1$. For technical reasons, let $\epsilon$ be a small positive number such that $\beta/\epsilon \in \mathbb{N}$.
We introduce intermediate scales
\begin{equation}
R^{0} \leq 
    R^{\epsilon} \leq R^{2\epsilon} \leq R^{3\epsilon} \leq \cdots \leq R^{(M-1)\epsilon} \leq R^{M\epsilon}
\end{equation}
where $M$ is a number such that $R^{M\epsilon}=R^{\beta}$. For simplicity, we introduce $R_k:=R^{k\epsilon}$. Note that $R_{M}=R^{\beta}$. Define $\mathcal{P}_{R_k^{-1}}$ to be a collection of $\gamma_k$ defined by
\begin{equation*}
\begin{split}
    \gamma_k:=\{(\xi_1,\xi_2,\xi_3) : l R_k^{-1} \leq &\xi_1 <(l+1)R_k^{-1}, 
    \\&|\xi_2-\xi_1^2| \leq R_k^{-2}, |\xi_3-3\xi_1\xi_2+2\xi_1^3| \leq \max \{R_k^{-3}, R^{-1} \} \}.
\end{split}
\end{equation*}
Note that $\gamma_M=\gamma$ where $\gamma$ is defined in \eqref{07.29.12}. Recall \eqref{24.11.11.111}.

\begin{definition}\label{24.01.25.def32}
For each $k=0,1,\ldots, \epsilon^{-1}$,
    define the square function
    \begin{equation}
        g_{k}(x):=\Big\|\sum_{\gamma_{k} \in \mP_{R_k^{-1}} }|f_{\gamma_{k}}|^2\Big\|_{L^1(w^{\#}_{B_{R_{k} }(x)})}, \;\;\; x \in \mathbb{R}^3.
    \end{equation}
    For each $k= \epsilon^{-1}+1,\ldots, M-1$, define
    \begin{equation}
        g_{k}(x):=\Big\|\sum_{\gamma_{k} \in \mP_{R_k^{-1}} }|f_{\gamma_{k}}|^2\Big\|_{L^1(w^{\#}_{B_{R_{k} }^{(2)}(x_1,x_2)} \times w^{\#}_{B_{R^{1+\epsilon}}^{(1)}(x_3)} )  }, \;\;\; x \in \mathbb{R}^3.
    \end{equation}
\end{definition}

\begin{definition}\label{07.30.def32}
Let $A$ be a sufficiently large constant.
The high set is defined by
\begin{equation}
    \Lambda_{M-1}:=\Big\{x \in B_{R^{2\beta}}^{(2)} \times B_R^{(1)} : A\sum_{\gamma \in \mPb}\|f_{\gamma}\|_{\infty}^2 \leq g_{M-1}(x) \Big\}.
\end{equation}
For each $k=1,\ldots,M-2$, we recursively define
\begin{equation}
    \Lambda_k:= \Big\{ x \in \big( B_{R^{2\beta}}^{(2)} \times B_R^{(1)} \big) \setminus \bigcup_{j=k+1}^{M-1} \Lambda_j : A^{M-k}\sum_{\gamma \in \mPb}\|f_{\gamma}\|_{\infty}^2 \leq g_k(x) \Big\}.
\end{equation}
Define the low set by
\begin{equation}
    \Lambda_0:= \big( B_{R^{2\beta}}^{(2)} \times B_R^{(1)} \big) \setminus \bigcup_{j=1}^{M-1} \Lambda_j.
\end{equation}
\end{definition}

\subsection{A high lemma}

Let us state elementary lemmas and definitions.

Let $\rho:\mathbb{R}^2 \rightarrow [0,\infty)$ be a smooth function such that
\begin{enumerate}
    \item $\mathrm{supp}(\rho) \subset [-10,10]^2$
    \item $\rho=1$ on the set $[-1,1]^2$
\end{enumerate}
Let $T_{k}$ be the linear transformation mapping $[-R_{k}^{-1},R_{k}^{-1}]^2$ to $[-1,1]^2$. Define
\begin{equation}
    {\rho}_{\leq R_k^{-1}}(\xi_1,\xi_2):={\rho}\big(T_{k}^{-1}(\xi_1,\xi_2) \big)
\end{equation}
so that the support of ${\rho}_{\leq R_k^{-1}}$ is contained in $[-10R_k^{-1},10R_k^{-1}]$. By a change of variables, we can see that $\|\widehat{\rho}_{ \leq R_k^{-1} }  \|_{L^1} \sim 1$.
For a Schwartz function $F: \R^3 \rightarrow \mathbb{C}$, define
\begin{equation}\label{07.29.310}
    F*\widehat{\rho}_{>R_k^{-1}}(x):=F(x)-F* \widehat{\rho}_{\leq R_k^{-1} }(x),
\end{equation}
where
\begin{equation*}
\begin{split}
    F* \widehat{\rho}_{\leq R_k^{-1} }(x)&
    =\int_{\R^2} F(x_1-y_1,x_2-y_2,x_3)\widehat{\rho}_{\leq R_k^{-1} }(y_1,y_2)\,dy_1dy_2.
\end{split}
\end{equation*}
This convolution is somewhat different from the one in the literature. The function $\hat{\rho}_{ \leq R_k^{-1}}$ is a function of varaibles $y_1,y_2$, so the convolution involves only two variables.
Lastly, we define the high part of the square function. For $k=0,1,\ldots,\epsilon^{-1}$,
\begin{equation}
    g_{k}^h(x_1,x_2,x_3):=\Big\|\sum_{\gamma_{k} \in \mP_{R_k^{-1}} }|f_{\gamma_{k}}|^2 * \widehat{\rho}_{>R_{k+1}^{-1}} \Big\|_{L^1(w^{\#}_{B_{R_{k} }(x)})}.
\end{equation}
For $k=\epsilon^{-1}+1,\ldots,M-1$,
\begin{equation}
    g_{k}^h(x_1,x_2,x_3):=\Big\|\sum_{\gamma_{k} \in \mP_{R_k^{-1}} }|f_{\gamma_{k}}|^2 * \widehat{\rho}_{>R_{k+1}^{-1}} \Big\|_{L^1(w^{\#}_{B_{R_{k} }^{(2)}(x_1,x_2)} \times w^{\#}_{B_{R^{1+\epsilon}}^{(1)}(x_3)} )  }.
\end{equation}
We have finished introducing the definitions. 

\begin{lemma}[cf. Corollary 1 of \cite{guth2022small}]\label{24.01.15.38}
Let $x \in \Lambda_{k}$ for some $k=1,\ldots,M-1$. Then we have
\begin{equation}
    g_{k}(x) \lesssim g_{k}^h(x).
\end{equation}    
\end{lemma}

\begin{proof} 
Let us consider the case that $k=M-1$.
Let $x \in \Lambda_{M-1}$. We write
\begin{equation}
\sum_{\gamma_{M-1}}|f_{\gamma_{M-1}}|^2    =\sum_{\gamma_{M-1}}|f_{\gamma_{M-1}}|^2 * \widehat{\rho}_{>R_{M}^{-1}}+
\sum_{\gamma_{M-1}}|f_{\gamma_{M-1}}|^2 * \widehat{\rho}_{\leq R_{M}^{-1}}. 
\end{equation}
By taking the integral on both sides, we have
\begin{equation*}
    g_{M-1}(x) \leq g_{M-1}^h(x)+
    \Big\|\sum_{\gamma_{M-1}}|f_{\gamma_{M-1}}|^2 * \widehat{\rho}_{\leq R_{M}^{-1}} \Big\|_{L^1(w^{\#}_{B_{R_{M-1} }^{(2)}(x_1,x_2)} \times w^{\#}_{B_{R^{1+\epsilon} }^{(1)}(x_3)})}.
\end{equation*}
For a contradiction, assume that
\begin{equation}\label{24.11.15.313}
    g_{M-1}(x) \lesssim  \Big\|\sum_{\gamma_{M-1}}|f_{\gamma_{M-1}}|^2 * \widehat{\rho}_{\leq R_{M}^{-1}} \Big\|_{L^1(w^{\#}_{B_{R_{M-1} }^{(2)}(x_1,x_2)} \times w^{\#}_{B_{R^{1+\epsilon} }^{(1)}(x_3)})}.
\end{equation}
By the $L^2$-orthogonality (cf. \cite[Lemma 4]{guth2022small}), for every $y \in \mathbb{R}^3$, we have
\begin{equation}
    \Big||f_{\gamma_{M-1}}|^2 * \widehat{\rho}_{\leq R_{M}^{-1}}(y)\Big| \lesssim 
\sum_{\gamma_M \subset \gamma_{M-1} } |f_{\gamma_{M}}|^2 * \big|\widehat{\rho}_{\leq R_{M}^{-1}}\big| (y).
\end{equation}
So, \eqref{24.11.15.313} is bounded by
\begin{equation}
    \lesssim \Big\|\sum_{\gamma_{M}}|f_{\gamma_{M}}|^2 * \big|\widehat{\rho}_{\leq R_{M}^{-1}}\big| \Big\|_{L^1(w^{\#}_{B_{R_{M-1} }^{(2)}(x_1,x_2)} \times w^{\#}_{B_{R^{1+\epsilon} }^{(1)}(x_3)})} \lesssim \sum_{\gamma_M}\|f_{\gamma_M}\|_{\infty}^2.
\end{equation}
In the last inequality, we used $\|\widehat{\rho}_{ \leq R_k^{-1} }  \|_{L^1} \sim 1$.
Since $R_M=R^{\beta}$, we have
\begin{equation}
    g_{M-1}(x) \leq C\sum_{\gamma \in \mPb }\|f_{\gamma}\|_{\infty}^2.
\end{equation}
By the definition of $\Lambda_{M-1}$ and $x \in \Lambda_{M-1}$,
\begin{equation}
    A\sum_{\gamma \in \mPb }\|f_{\gamma}\|_{\infty}^2 \leq g_{M-1}(x).
\end{equation}
    Since $A$ is a sufficiently large number, this gives a contradiction.

    Let us next consider the case that $\epsilon^{-1}+1 \leq k \leq M-2$. By the definition of $\Lambda_{k}$ and $x \in \Lambda_k$, we have
    \begin{equation}\label{24.01.18.315}
        A^{M-k-1} \sum_{\gamma}\|f_{\gamma}\|_{\infty}^2 > g_{k+1}(x), \;\;\;\; A^{M-k}\sum_{\gamma}\|f_{\gamma}\|_{\infty}^2 \leq g_k(x).
    \end{equation}
    By following the arguments for the case $k=M-1$, we have
    \begin{equation}
        g_{k}(x) \leq C g_{k+1}(x) < CA^{M-k-1} \sum_{\gamma}\|f_{\gamma}\|_{\infty}^2.
    \end{equation}
    By the second inequality of \eqref{24.01.18.315}, we have
    \begin{equation}
        A^{M-k} <CA^{M-k-1}.
    \end{equation}
This gives a contradiction as $A$ is sufficiently large. The case that $0 \leq k \leq \epsilon^{-1}$ can be dealt with identically. We omit the details.
\end{proof}

A key property of the function $g_k^h$ is that the Fourier supports of $\{|f_{\gamma_k}|^2 * \hat{\rho}_{>R_{k+1}^{-1}} \}_k$ are essentially disjoint. Combined with the $L^2$ orthogonality we have the high lemma. To simplify the notation, we write $x=(x',x_3) \in \mathbb{R}^2 \times \mathbb{R}$.

\begin{lemma}[High lemma]\label{24.01.15.lem35}

Let  $x=(x',x_3) \in \R^2 \times \R$. For $k=0,1,\ldots,\epsilon^{-1}$,
    \begin{equation*}
        \begin{split}
    &\int_{B^{(2)}_{(R_{k+1})^2}(x')}|g_{k}^h(y',x_3)|^2\,dy' \\&\lesssim R^{O(\epsilon)} \int_{\R^2} \Big\| \sum_{\gamma_{k} \in \mP_{R_k^{-1}} }
        \big||f_{\gamma_{k}}|^2 * \widehat{\rho}_{>R_{k+1}^{-1}} \big|^2\Big\|_{L^1(w_{B_{R_k}(y',x_3) }^{\#} )} w_{B_{(R_{k+1})^2}^{(2)}(x') }(y') \,dy'.
    \end{split}
    \end{equation*}
For $k=\epsilon^{-1}+1,\ldots,M-1$,
\begin{equation*}
     \begin{split}
    &\int_{B^{(2)}_{(R_{k+1})^2}(x')}|g_{k}^h(y',x_3)|^2\,dy' \\&\lesssim R^{O(\epsilon)} \int_{\R^2} \Big\| \sum_{\gamma_{k} \in \mP_{R_k^{-1}} }
        \big||f_{\gamma_{k}}|^2 * \widehat{\rho}_{>R_{k+1}^{-1}} \big|^2\Big\|_{L^1(w_{B_{R_k}^{(2)}(y') }^{\#} \times  w_{B_{R^{1+\epsilon}}^{(1)}(x_3) }^{\#}  )} w_{B_{(R_{k+1})^2}^{(2)}(x') }(y') \,dy'.
    \end{split}
\end{equation*}
\end{lemma}

\begin{proof}
Let us first consider the case that $k=0,1,\ldots,\epsilon^{-1}$.
    By H\"{o}lder's inequality,
    \begin{equation}
    \begin{split}
        &\int_{B_{(R_{k+1})^2}^{(2)}(x')}|g_{k}^h(y',x_3)|^2\,dy' \\& \lesssim
        \int_{B_{(R_{k+1})^2}^{(2)}(x')} \int_{\R^3}
        \Big|
        \sum_{\gamma_{k}}|f_{\gamma_{k}}|^2 * \widehat{\rho}_{>R_{k+1}^{-1}}(z) \Big|^2
        w^{\#}_{B_{R_{k} }(y',x_3)}(z)\,dz\,dy'.
        \end{split}
    \end{equation}
By a change of variables on $z$, this can be rewritten as 
\begin{equation*}
    \int_{B_{(R_{k+1})^2}^{(2)}(x')} \int_{\R^3}
        \Big|
        \sum_{\gamma_{k}}|f_{\gamma_{k}}|^2 * \widehat{\rho}_{>R_{k+1}^{-1}}((y',x_3)-z) \Big|^2
        w^{\#}_{B_{R_{k} }(0)}(z)\,dz \, dy'.
\end{equation*}
By Fubini's theorem, this is equal to
\begin{equation*}
     \int_{\R^3} \Big(\int_{B_{(R_{k+1})^2}^{(2)}(x')}
        \Big|
        \sum_{\gamma_{k}}|f_{\gamma_{k}}|^2 * \widehat{\rho}_{>R_{k+1}^{-1}}((y',x_3)-z) \Big|^2 \, dy' \Big)
        w^{\#}_{B_{R_{k} }(0)}(z)\,dz.
\end{equation*}
By Plancherel's theorem on the variables $y'$, and disjointness of Fourier supports, this is bounded by
\begin{equation*}
\begin{split}
\lesssim R^{O(\epsilon)}
    \int_{\R^3} \int_{\mathbb{R}^2}
        \sum_{\gamma_{k}}
        \Big|
        |f_{\gamma_{k}}|^2 * \widehat{\rho}_{>R_{k+1}^{-1}}((y',x_3)-z) \Big|^2
        w_{B_{(R_{k+1})^2}^{(2)}(x') }(y') \, dy' 
        \,w^{\#}_{B_{R_{k} }(0)}(z)\,dz.
    \end{split}
\end{equation*}
By Fubini's theorem and a change of variables on $z$, this is equal to
\begin{equation*}
    \int_{\R^2} \int_{\mathbb{R}^3}
        \sum_{\gamma_{k}}
        \Big|
        |f_{\gamma_{k}}|^2 * \widehat{\rho}_{>R_{k+1}^{-1}}(z) \Big|^2  
        \,w^{\#}_{B_{R_{k} }(0)}((y',x_3)-z)\,dz w_{B_{(R_{k+1})^2}^{(2)}(x') }(y') \, dy'.
\end{equation*}
This can be rewritten as
\begin{equation*}
\int_{\R^2} \Big\| \sum_{\gamma_{k} \in \mP_{R_k^{-1}} }
        \big||f_{\gamma_{k}}|^2 * \widehat{\rho}_{>R_{k+1}^{-1}} \big|^2\Big\|_{L^1(w_{B_{R_k}(y',x_3) }^{\#} )} w_{B_{(R_{k+1})^2}^{(2)}(x') }(y') \,dy'.
    \end{equation*}
This completes the proof for the case $k=0,\ldots,\epsilon^{-1}$. The case $k=\epsilon^{-1}+1,\ldots,M-1$ can be proved identically. We omit the proof here.
\end{proof}

We combine Lemma \ref{24.01.15.38} and \ref{24.01.15.lem35}, and obtain the following.

\begin{corollary}\label{24.01.15.cor36} For $k=0,1,\ldots,\epsilon^{-1}$, we have
    \begin{equation}\label{24.03.02}
        \sum_{B_{R_k}: B_{R_k} \cap \Lambda_k \neq \emptyset }\int_{B_{R_k}}|g_{k}|^2 \lesssim \int_{\R^3} \sum_{\gamma_{k}}|f_{\gamma_{k}}|^4.
    \end{equation}
    Similarly, for $k=\epsilon^{-1}+1,\ldots,M-1$, we have
    \begin{equation}\label{24.03.26}
    \sum_{(B_{R_{k}}^{(2)} \times B_{R^{1+\epsilon}}^{(1)} ) \cap \Lambda_k \neq \emptyset}  \int_{B_{R_{k}}^{(2)} \times B_{R^{1+\epsilon}}^{(1)} }  |g_k|^2 \lesssim \int_{\R^3} \sum_{\gamma_{k}}|f_{\gamma_{k}}|^4.
\end{equation}
\end{corollary}

\begin{proof}
Let us prove \eqref{24.03.02}. 
    Suppose that $B_{R_k} \cap \Lambda_k \neq \emptyset$. Take $x_0 \in B_{R_k} \cap \Lambda_k$. For any $x \in B_{R_k}$, by a property of $w^{\#}_{B_{R_k}}$, we have
    \begin{equation}
        g_k(x) \lesssim g_k(x_0) \lesssim g_k^h(x_0) \lesssim g_k^h(x).
    \end{equation}
    The second inequality follows by Lemma \ref{24.01.15.38}. So we have
    \begin{equation}
    \begin{split}
        \sum_{B_{R_k}: B_{R_k} \cap \Lambda_k \neq \emptyset }\int_{B_{R_k}}|g_{k}(x)|^2\,dx &\lesssim \int_{N_{(R_k)^2}(\Lambda_k)} |g_k^h(x)|^2\,dx
        \\& \lesssim \int_{N_{(R_{k+1})^2}(\Lambda_k)} |g_k^h(x)|^2\,dx.
    \end{split}
    \end{equation}
By applying Lemma \ref{24.01.15.lem35} to a ball of radius $(R_{k+1})^2$, this is further bounded by
\begin{equation*}
\sum_{B_{R_{k+1}}^{(2)} }
       \int_{\R^3} \Big\| \sum_{\gamma_{k} \in \mP_{R_k^{-1}} }
        \big||f_{\gamma_{k}}|^2 * \widehat{\rho}_{>R_{k+1}^{-1}} \big|^2\Big\|_{L^1(w_{B_{R_{k}}(y',x_3) }^{\#} )} w_{B_{R_{k+1}}^{(2)}} (y')   \,dy'dx_3,
    \end{equation*}
    By the inequality $\sum_{B} w_{B} \lesssim 1$, we have
    \begin{equation*}
       \int_{\R^3} \int_{\R^3} \sum_{\gamma_{k} \in \mP_{R_k^{-1}} }
        \big||f_{\gamma_{k}}|^2 * \widehat{\rho}_{>R_{k+1}^{-1}} (z)\big|^2(w_{B_{R_k}(y',x_3) }^{\#}(z) )   \, dz \, dy' dx_3.
    \end{equation*}
By calculating the integral over $dy'dx_3$ first, this is bounded by
\begin{equation*}
       \int_{\R^3} \sum_{\gamma_{k} \in \mP_{R_k^{-1}} }
        \big||f_{\gamma_{k}}|^2 * \widehat{\rho}_{>R_{k+1}^{-1}} (z)\big|^2\,dz.
    \end{equation*}
By \eqref{07.29.310}, we have
\begin{equation}
    \big||f_{\gamma_{k}}|^2 * \widehat{\rho}_{>R_{k+1}^{-1}}(z) \big|^2 \lesssim |f_{\gamma_{k}}(z)|^4 + \big||f_{\gamma_{k}}|^2 * \widehat{\rho}_{ \leq R_{k+1}^{-1}}(z) \big|^2.
\end{equation}
Since $\|\widehat{\rho}_{ \leq R_{k+1}^{-1}} \|_{L^1} \lesssim 1$, by Young's inequality, this is bounded by
    \begin{equation}
        \int_{\R^3} \sum_{\gamma_{k} \in \mP_{R_k^{-1}} } |f_{\gamma_k}|^4.
    \end{equation}
    This completes the proof. \eqref{24.03.26} can be proved identically, and we omit the proof.
\end{proof}

We finish this subsection by stating the Bernstein's inequality.

\begin{lemma}[cf. Lemma 2 of \cite{guth2022small}]\label{24.01.15.lem33}
Let $f$ be a function whose Fourier support is in an ellipsoid $A \subset \mathbb{R}^3$. Then for any dual set $A^*$, we have
    \begin{equation}
        \|f\|_{L^{\infty}(A^*)} \lesssim \|f\|_{L^1(w^{\#}_{A^*})}.
    \end{equation}
\end{lemma}

\subsection{Multilinear restriction/Kakeya inequalities}

We say $I_1,I_2,I_3 \subset [0,1]$ are $C$-transverse if $\mathrm{dist}(I_i,I_j)>C$ for any $i \neq j$. The following theorem is a multlilinear restrictiton estimate for the moment curve. The proof is a simple application of Plancherel's theorem (for example, see  \cite[Lemma 2.5]{MR3161099}).

\begin{proposition}[Multilinear restriction estimate]\label{24.01.19.prop37}

Let $C>0$ and $\beta \geq 1$.
Let $I_1,I_2,I_3 \subset [0,1] $ be $C$-transverse. Then
\begin{equation}
    \Big\| \avprod_{i=1}^3 f_i \Big\|_{L_{\#}^6(B_R)} \lesssim_C \avprod_{i=1}^3\| f_i \|_{L^2(w_{B_R}^{\#})}
\end{equation}
    for any functions $f_i$ whose Fourier support is in $\mathcal{M}^3(R^{\beta},R) \cap (I_i \times \mathbb{R}^2)$.
\end{proposition}

We next state the bilinear Kakeya inequality. To state the inequality, let us introduce some notations.
Let $\mathbb{T}^{(2)}$ be a collection of tubes in $\R^2$ with dimension $R \times R^2$ contained in a ball of radius $10R^2$. We say $\mathbb{T}_1,\mathbb{T}_2 \subset \mathbb{T}^{(2)}$ are $\nu$-separated if for any tubes $T_1 \in \mathbb{T}_1$ and $T_2 \in \mathbb{T}_2$ 
\begin{equation}
    \angle(T_1,T_2) \geq \nu >0
\end{equation}
where $\angle(T_1,T_2)$ is the angle between
the long directions of $T_1$ and $T_2$.

\begin{lemma}[Bilinear Kakeya inequality]\label{24.02.04.prop38}

Let $\nu>0$ and $R > 1$. Suppose that $\mathbb{T}_1,\mathbb{T}_2$ are $\nu$-separated. Then
\begin{equation}
    \fint_{B_{R^2}^{(2)}} \sum_{T_1} \sum_{T_2}\chi_{T_1}\chi_{T_2} \lesssim_\nu \fint_{B_{R^2}^{(2)}} \sum_{T_1}\chi_{T_1} \fint_{B_{R^2}^{(2)}} \sum_{T_2}\chi_{T_2}.
\end{equation}

\end{lemma}

\begin{proof}
    By linearlity, it suffices to prove that
    \begin{equation}
    \fint_{B_{R^2}^{(2)}} \chi_{T_1}\chi_{T_2} \lesssim_\nu \fint_{B_{R^2}^{(2)}} \chi_{T_1} \fint_{B_{R^2}^{(2)}} \chi_{T_2}.
\end{equation}
By routine computations,
\begin{equation}
    \fint_{B_{R^2}^{(2)}} \chi_{T_1}\chi_{T_2} \sim R^{-4}|T_1 \cap T_2| \sim R^{-2}
\end{equation}
and
\begin{equation}
    \fint_{B_{R^2}^{(2)}} \chi_{T_i} \sim R^{-4}|T_i| \sim R^{-1}.
\end{equation}
This completes the proof.
\end{proof}

\section{Proof of Theorem \ref{23.12.31.thm12}}

In this section, we give a proof of Theorem \ref{23.12.31.thm12}. Recall that Theorem \ref{23.12.31.thm12} follows from Proposition \ref{24.01.18.prop23}. Recall the decomposition (Definition \ref{07.30.def32})
\begin{equation}
    B_{R^{2\beta}}^{(2)} \times B_R^{(1)} = \bigcup_{k=0}^{M-1}\Lambda_k.
\end{equation}
Let $p_c:=6+2/\beta$.
Since $M = \beta/\epsilon$, it suffices to prove
\begin{equation}\label{24.01.19.42}
\alpha^{p_c}|U_{\alpha,\mathrm{mul} } \cap \Lambda_{k} | \leq C_{\epsilon} R^{p_c \beta(\frac12-\frac{1}{p_c})} R^{\epsilon} \sum_{\gamma} \|f_{\gamma}\|_{L^4}^4 \big( \sup_{\gamma} \|f_{\gamma}\|_{L^{\infty}} \big)^{{p_c-4}}
    \end{equation}
for all $k=0,1,\ldots,M-1$.
\begin{lemma}\label{24.01.18.lem411}
Let $0 \leq k \leq M-1$.
Suppose that $U_{\alpha,\mathrm{mul}} \cap \Lambda_{k} \neq \emptyset$. Then
    \begin{equation}\label{24.01.18.42}
        \alpha \lesssim R^{O(\epsilon)} (R_{k})^{\frac12}(\sum_{\gamma}\|f_{\gamma}\|_{\infty}^2)^{\frac12}.
    \end{equation}
\end{lemma}

\begin{proof}
Suppose $0 \leq k \leq M-2$.
Let $x \in U_{\alpha,\mathrm{mul}} \cap \Lambda_{k} \neq \emptyset$. Then we have
\begin{equation}
    \alpha \lesssim |f_1(x)f_2(x)f_3(x)|^{\frac13} \lesssim \sum_{\gamma_{k}}|f_{\gamma_{k}}(x)|.
\end{equation}
Since $\# \gamma_k \lesssim R_k$,
by Cauchy-Schwarz inequality, this is bounded by
\begin{equation}
    (R_k)^{\frac12} \big(\sum_{\gamma_k}|f_{\gamma_k}(x)|^2\big)^{\frac12}.
\end{equation}
By an application of Fourier uncertainty principle, and the definition of $\Lambda_k$,
this is further bounded by 
\begin{equation}
    \lesssim (R_k)^{\frac12}g_k(x)^{\frac12} \lesssim R^{O(\epsilon)}
    (R_k)^{\frac12}g_{k+1}(x)^{\frac12} \lesssim R^{O(\epsilon)}(R_k)^{\frac12}(\sum_{\gamma}\|f_{\gamma}\|_{\infty}^2)^{\frac12}.
\end{equation}
This completes the proof. The case $k=M-1$ can be done similarly.
\end{proof}

One ingredient of the proof of Theorem \ref{23.12.31.thm12} is the $L^4$ reverse square function estimate for the parabola. As a corollary of it, we can obtain the following inequality.

\begin{lemma}\label{24.02.13.lem42}
For every function $f$ whose Fourier support is in $\mathcal{M}^3(R^{2},R)$, and for every $c \in \R$, we have
    \begin{equation}
        \|f\|_{L^4([0,R^2]^2 \times \{c \})} \leq C_p  \big\| (\sum_{\gamma \in \mP_{R^{-1}} }|f_{\gamma}|^2)^{\frac12} \big\|_{L^4(\R^2 \times \{c\}) }.
    \end{equation}
\end{lemma}

We will estimate $|U_{\alpha,\mathrm{mul}} \cap \Lambda_{k}|$. Our proofs are different whether
\begin{equation}
    k=0, \;\;\; 1 \leq k < \epsilon^{-1}, \;\;\; \mathrm{or} \;\;\; \epsilon^{-1} \leq k \leq \beta\epsilon^{-1}.
\end{equation}
Let us start with the last case, which is the nontrivial case.

\subsection{Estimate for  \texorpdfstring{$|U_{\alpha,\mathrm{mul}} \cap \Lambda_{k}|$}{} for \texorpdfstring{$\epsilon^{-1} \leq k \leq \beta \epsilon^{-1}$}{}} We may assume that $U_{\alpha,\mathrm{mul}} \cap \Lambda_{k} \neq \emptyset$.
By the definition of $U_{\alpha,\mathrm{mul}}$, we have the following inequality.
\begin{equation}
    \alpha^6|U_{\alpha,\mathrm{mul}} \cap \Lambda_{k}| \leq
    \sum_{B_{R}: B_{R} \cap U_{\alpha,\mathrm{mul}} \cap \Lambda_k \neq \emptyset }
    \int_{B_{R}} \Big|\avprod_{i=1}^3 f_i \Big|^6.
\end{equation}
We apply Proposition \ref{24.01.19.prop37} and H\"{o}lder's inequality, and this is bounded by
\begin{equation}\label{07.30.410}
\begin{split}
    \lesssim 
    \sum_{B_{R}: B_{R} \cap U_{\alpha,\mathrm{mul}} \cap \Lambda_{k} \neq \emptyset }
    \avprod_{i=1}^3
    \int_{\mathbb{R}^3} \big( \sum_{\theta \in \mathcal{P}_{R^{-1}} : \theta \subset I_i }|f_{\theta}(x)|^2 \big)^3 w_{B_{R}}(x)\,dx.
    \end{split}
\end{equation}
Since for given $\theta \in \mP_{R^{-1}}$ the number of $\gamma_k \subset \theta$ is bounded by $R_kR^{-1}$,
we have
\begin{equation}
    \sum_{\theta \in \mathcal{P}_{R^{-1}}}|f_{\theta}(x)|^2 \lesssim R_kR^{-1} \sum_{\gamma_k }|f_{\gamma_k}(x)|^2.
\end{equation}
Let $y \in B_{R} \cap U_{\alpha,\mathrm{mul}} \cap \Lambda_{k} \neq \emptyset$. Let $B_{R^{1+\epsilon}}$ be the a $CR^{1+\epsilon}$-neighborhood of the ball $B_R$. Note that by a rapid decay of $w_{B_R}$ outside of the ball $B_{R^{1+\epsilon}}$, the main contribution of \eqref{07.30.410} comes from the integral over $B_{R^{1+\epsilon}}$.  For $x \in B_{R^{1+\epsilon}}$, by Fourier uncertainty principle, we have
\begin{equation}
    \sum_{\gamma_k }|f_{\gamma_k}(x)|^2 \lesssim R^{O(\epsilon)} 
    g_{k+1}(y) \lesssim R^{O(\epsilon)}\sum_{\gamma}\|f_{\gamma}\|_{\infty}^2.
\end{equation}
Combining these two, we have
\begin{equation}
    \sum_{\theta \in \mathcal{P}_{R^{-1}}}|f_{\theta}(x)|^2 \lesssim R^{O(\epsilon)}R_kR^{-1}\sum_{\gamma}\|f_{\gamma}\|_{\infty}^2
\end{equation}
for any $x \in B_{R^{1+\epsilon}}$. By applying this inequality to \eqref{07.30.410}, we have
\begin{equation}\label{24.01.19.4112}
\begin{split}
& \alpha^6|U_{\alpha,\mathrm{mul}} \cap \Lambda_{k}|
 \\  & \lesssim R^{O(\epsilon)} R_kR^{-1}\sum_{\gamma}\|f_{\gamma}\|_{\infty}^2
    \sum_{B_{R}: B_{R} \cap U_{\alpha,\mathrm{mul}} \cap \Lambda_{k} \neq \emptyset }
    \avprod_{i=1}^3
    \int_{ B_{R^{1+\epsilon}}  } \big( \sum_{\theta \in \mathcal{P}_{R^{-1}} : \theta \subset I_i }|f_{\theta}|^2 \big)^2.
    \end{split}
\end{equation}
For any sequences $A_i,B_i,C_i$, by H\"{o}lder's inequality, we have
\begin{equation}
\sum_i
    (A_iB_iC_i)^{\frac13} \lesssim \big(\sum (A_iB_i)^{\frac12} \big)^{\frac13} 
    \big(\sum (A_iC_i)^{\frac12} \big)^{\frac13} 
    \big(\sum (B_iC_i)^{\frac12} \big)^{\frac13}. 
\end{equation}
Hence, by abusing a notation,  we may say that \eqref{24.01.19.4112} is bounded by
\begin{equation}\label{24.02.13.415}
    \lesssim R^{O(\epsilon)} R_kR^{-1}\sum_{\gamma}\|f_{\gamma}\|_{\infty}^2
    \sum_{B_{R}: B_{R} \cap U_{\alpha,\mathrm{mul}} \cap \Lambda_{k} \neq \emptyset }
    \avprod_{i=1}^2
    \int_{ B_{R^{1+\epsilon}} } \big( \sum_{\theta \in \mathcal{P}_{R^{-1}} : \theta \subset I_i }|f_{\theta}|^2 \big)^2.
\end{equation}
We will show that
\begin{equation}\label{24.02.13.416}
\begin{split}
    \sum_{B_{R}: B_{R} \cap U_{\alpha,\mathrm{mul}} \cap \Lambda_{k} \neq \emptyset }
    \avprod_{i=1}^2
    \int_{ B_{R^{1+\epsilon}} } \big( \sum_{\theta \in \mathcal{P}_{R^{-1}} : \theta \subset I_i }|f_{\theta}|^2 \big)^2  
    \lesssim \int_{\R^3}  \sum_{\gamma_k}|f_{\gamma_k}|^4.
\end{split}
\end{equation}
Assuming this, let us finish the proof of \eqref{24.01.19.42}. By \eqref{24.01.19.4112}, \eqref{24.02.13.415}, \eqref{24.02.13.416}, we have
\begin{equation}
\begin{split}
    \alpha^6|U_{\alpha,\mathrm{mul}} \cap \Lambda_{k}| &\lesssim R^{O(\epsilon)} R_kR^{-1}\sum_{\gamma}\|f_{\gamma}\|_{\infty}^2 \int \sum_{\gamma_k}|f_{\gamma_k}|^4
    \\& \lesssim  R^{O(\epsilon)}R_kR^{-1}R^{\beta}
    \sup_{\gamma}\|f_{\gamma}\|_{\infty}^2 \int \sum_{\gamma_k}|f_{\gamma_k}|^4.
\end{split}
\end{equation}
Note that $\#(\gamma \subset \gamma_k) \lesssim R^{\beta}/R_k$. So
by Lemma \ref{24.02.13.lem42} and H\"{o}lder's inequality, this is further bounded by
\begin{equation}
\begin{split}
    &\lesssim R^{O(\epsilon)}R_kR^{-1}R^{\beta}
    (\frac{R^{\beta}}{R_k})\sup_{\gamma}\|f_{\gamma}\|_{\infty}^2 \int \sum_{\gamma}|f_{\gamma}|^4
    \\& \sim R^{O(\epsilon)}R^{-1}R^{2\beta}\sup_{\gamma}\|f_{\gamma}\|_{\infty}^2 \int \sum_{\gamma}|f_{\gamma}|^4.
\end{split}
\end{equation}
Recall that $p_c=6+2/\beta$ and
\begin{equation}
    p_c\beta(\frac12-\frac{1}{p_c})=1+2\beta.
\end{equation}
Hence, to prove \eqref{24.01.19.42}, what we need to check is
\begin{equation}\label{24.02.13.421}
    \alpha^{2/\beta} \lesssim R^{O(\epsilon)} R^2 \big(\sup_{\gamma}\|f_{\gamma}\|_{\infty} \big)^{\frac{2}{\beta}}.
\end{equation}
By Lemma \ref{24.01.18.lem411},
\begin{equation}
    \alpha \lesssim R^{O(\epsilon)}(R_kR^{\beta})^{\frac12}\sup_{\gamma}\|f_{\gamma}\|_{\infty} \lesssim R^{O(\epsilon)}R^{\beta}\sup_{\gamma}\|f_{\gamma}\|_{\infty},
\end{equation}
and this verifies \eqref{24.02.13.421}.
\\

We have proved \eqref{24.01.19.42} under the assumption of \eqref{24.02.13.416}. It remains to prove \eqref{24.02.13.416}.
To proceed, we pass by the intermediate scale ${R_{k}}$.
We claim that 
\begin{equation}\label{24.01.25.415}
\begin{split}
    \sum_{B_{R} \subset B_{R_{k}}^{(2)} \times B_R^{(1)} }
    \Big(\avprod_{i=1}^2
    \int_{ B_{R^{1+\epsilon}} } \big( \sum_{\theta \in \mathcal{P}_{R^{-1}} : \theta \subset I_i }|f_{\theta}|^2 \big)^2  \Big) \lesssim R^{O(\epsilon)} \int_{B_{R_{k}}^{(2)} \times B_{R^{1+\epsilon}}^{(1)} } |g_{k}|^2.
\end{split}
\end{equation}
Let us assume the claim and continue to prove \eqref{24.02.13.416}. By \eqref{24.01.25.415},
\begin{equation}\label{24.02.14.430}
\begin{split}
    &\sum_{B_{R}: B_{R} \cap U_{\alpha,\mathrm{mul}} \cap \Lambda_{k} \neq \emptyset }
    \avprod_{i=1}^2
     \int_{ B_{R^{1+\epsilon}} } \big( \sum_{\theta \in \mathcal{P}_{R^{-1}} : \theta \subset I_i }|f_{\theta}|^2 \big)^2  
    \\&
    \lesssim \sum_{(B_{R_{k}}^{(2)} \times B_{R^{1+\epsilon} }^{(1)}) \cap \Lambda_k \neq \emptyset } \int_{B_{R_{k}}^{(2)} \times B_{R^{1+\epsilon}}^{(1)}} |g_{k}|^2.
\end{split}
\end{equation}
By Corollary \ref{24.01.15.cor36}, this is bounded by
\begin{equation}
    \int_{\R^3} \sum_{\gamma_{k}}|f_{\gamma_{k}}|^4.
\end{equation}
This completes the proof of \eqref{24.02.13.416}.

Let us give a proof of the claim. By H\"{o}lder's inequality, we have
\begin{equation}\label{24.02.04.417}
    \begin{split}
        \int_{B_{R^{1+\epsilon}}} \big( \sum_{\theta \in \mathcal{P}_{R^{-1}} : \theta \subset I_i }|f_{\theta}|^2 \big)^2
        \lesssim
          R^{O(\epsilon)}\int_{B_{R^{1+\epsilon}}} \big( \sum_{\theta' \in \mathcal{P}_{R^{-1-\epsilon}} : \theta' \subset I_i }|f_{\theta'}|^2 \big)^2.
    \end{split}
\end{equation}
By the definition of $\theta' \in \mP_{R^{-1-\epsilon}}$, we have
\begin{equation*}
\begin{split}
    \theta'=\{(\xi_1,\xi_2,\xi_3) : l R^{-1-\epsilon} \leq &\xi_1 <(l+1)R^{-1-\epsilon}, 
    \\&|\xi_2-\xi_1^2| \leq R^{-2-2\epsilon}, |\xi_3-3\xi_1\xi_2+2\xi_1^3| \leq R^{-1} \}
\end{split}
\end{equation*}
for some $l$. This set is contained in a rectangular box with dimension $R^{-1-\epsilon} \times R^{-2-\epsilon} \times R^{-1}$, and long directions are parallel to the $\xi_3$-axis. Then by an uncertainty principle, on the set $B_{R^{2+2\epsilon}}^{(2)} \times B_{R}^{(1)}$, we can write
\begin{equation}\label{24.11.15.425}
    |f_{\theta'}(x_1,x_2,x_3)|^2=\sum_{T_{\theta'}} |a_{\theta',i}|\chi_{T_{\theta'}}(x_1,x_2,x_3),
\end{equation}
where $T_{\theta'}$ is a rectangular box with dimension $R^{1+\epsilon} \times R^{2+2\epsilon} \times R$ and the shortest direction is parallel to the $x_3$-axis. For simplicity, let us first consider the case that $R_k \leq R^{2+2\epsilon}$ as a toy case. Then
\begin{equation}\label{24.02.04.419}
\begin{split}
    &\sum_{B_{R} \subset B_{R_{k}}^{(2)} \times B_R^{(1)} }
    \Big(\avprod_{i=1}^2
    \int_{ B_{R^{1+\epsilon}} } \big( \sum_{\theta' \in \mathcal{P}_{R^{-1-\epsilon}} : \theta' \subset I_i }|f_{\theta'}|^2 \big)^2  \Big)
    \\&\lesssim R^{O(\epsilon)} R
    \sum_{B_{R^{1+\epsilon}}^{(2)} \subset B_{R_{k}}^{(2)} }
    \int_{B_{R^{1+\epsilon}}^{(2)}}
    \big(\avprod_{i=1}^2 \sum_{\theta', T_{\theta'}}|a_{\theta',i}|\chi_{T_{\theta'}}(x_1,x_2,0) \big)^2 \,dx_1dx_2
    \\&\lesssim  R^{O(\epsilon)} R
    \int_{B_{R_{k}}^{(2)}}\big(
    \avprod_{i=1}^2 \sum_{\theta', T_{\theta'}}|a_{\theta',i}|^2\chi_{T_{\theta'}}(x_1,x_2,0) \big)^2 \,dx_1dx_2.
\end{split}
\end{equation}
For convenience, we write the last inequality by
\begin{equation}\label{07.30.427}
    R^{O(\epsilon)} R(R_{k})^2\fint_{B_{R_{k}}^{(2)}}\big(
    \avprod_{i=1}^2 \sum_{\theta', T_{\theta'}}|a_{\theta',i}|^2\chi_{T_{\theta'}}(x_1,x_2,0) \big)^2dx_1dx_2.
\end{equation}
Note that the set
\begin{equation}
    \{ (x_1,x_2): (x_1,x_2,0) \in T_{\theta'} \} \cap B_{R_k}^{(2)}
\end{equation}
 is a tube with dimension $R_{k} \times R^{1+\epsilon}$. We do isotropic rescaling so that the tube becomes a tube with dimension $R_{k}R^{-1-\epsilon} \times 1$. Then apply Lemma \ref{24.02.04.prop38}, and rescale back. After this process, \eqref{07.30.427} is bounded by
\begin{equation}
 R^{O(\epsilon)} R
    (R_{k})^2
\big(\avprod_{i=1}^2\fint_{B_{R_{k}}^{(2)}} \sum_{\theta' \subset I_i , T_{\theta'}}|a_{\theta',i}|\chi_{T_{\theta'}}\big)^2.
\end{equation}
This is further bounded by
\begin{equation}\label{24.02.04.420}
\begin{split}
&\lesssim   R^{O(\epsilon)} R(R_{k})^2\big(\avprod_{i=1}^2\fint_{B_{R_{k}}^{(2)}} \sum_{\theta' \subset I_i }|f_{\theta'}|^2\big)^2 \\&
\lesssim   R^{O(\epsilon)} R(R_{k})^2 \big( \avprod_{i=1}^2 \int_{\R^2} \sum_{\gamma_{k} \in \mP_{R_{k}^{-1}}: \gamma_k \subset I_i }|f_{\gamma_{k}}|^2 w^{\#}_{B_{R_{k}}} \big)^2  \lesssim   R^{O(\epsilon)} R (R_{k})^2 \big(g_{k}(x_0)\big)^2,
\end{split}
\end{equation}
where $x_0$ is any point in $B_{R_{k}}^{(2)} \times B_{R^{1+\epsilon} }^{(1)}$. So this is bounded by
\begin{equation}
    \int_{B_{R_{k}}^{(2)} 
 \times B_{R^{1+\epsilon}} }|g_{k}(x_1,x_2,x_3)|^2 \, dx_1 dx_2 dx_3.
\end{equation}
Combining all the inequalities gives the claim \eqref{24.01.25.415}. We now suppose that $R_k >R^{2+2\epsilon}$. Then on \eqref{24.02.04.419}, we write
\begin{equation}
    \sum_{B_R \subset B_{R_k}^{(2)} \times B_R^{(1)} } =\sum_{B_{R^{2+2\epsilon}}^{(2)} \times B_R^{(1)} \subset B_{R_k}^{(2)} \times B_R^{(1)} }\sum_{B_R \subset B_{R^{2+2\epsilon}}^{(2)} \times B_R^{(1)} }.
 \end{equation}
 By repeating the previous argument, \eqref{24.02.04.419} is bounded by
 \begin{equation}
     R^{O(\epsilon)} R(R^{2+2\epsilon} )^2\sum_{B_{R^{2+2\epsilon}}^{(2)}  \subset B_{R_k}^{(2)}  }
      \big( \avprod_{i=1}^2 \int_{\R^2} \sum_{\gamma_{k_0} \in \mP_{R^{-2-2\epsilon}}: \gamma_{k_0} \subset I_i }|f_{\gamma_{k_0}}|^2 w^{\#}_{B_{R^{2+2\epsilon}}} \big)^2
 \end{equation}
 where $R_{k_0}=R^{2+2\epsilon}$. We repeat \eqref{24.11.15.425} on the set $B_{R^{4+4\epsilon}}^{(2)} \times B_R^{(1)}$, and follow the previous argument. Repeat the whole argument until we reach the scale $R_k$. Then we get the desired bound. This bootstrapping argument appears in \cite{MR2275834}, so we leave out the details.
\medskip

\subsection{Estimate for  \texorpdfstring{$|U_{\alpha,\mathrm{mul}} \cap \Lambda_{k}|$}{} for \texorpdfstring{$1 \leq k \leq \epsilon^{-1}$}{}}

We start with

\begin{equation}\label{24.01.19.47}
    \alpha^6|U_{\alpha,\mathrm{mul}} \cap \Lambda_{k}| \leq
    \sum_{B_{R_{k-1}}: B_{R_{k-1}} \cap U_{\alpha,\mathrm{mul}} \cap \Lambda_k \neq \emptyset }
    \int_{B_{R_{k-1}}} \Big|\avprod_{i=1}^3 f_i \Big|^6.
\end{equation}
We apply Proposition \ref{24.01.19.prop37}, and this gives
\begin{equation}
\lesssim 
    \sum_{B_{R_{k-1}}: B_{R_{k-1}} \cap U_{\alpha,\mathrm{mul}} \cap \Lambda_k \neq \emptyset }
    \int_{\mathbb{R}^3} \big( \sum_{\gamma_{k-1}}|f_{\gamma_{k-1}}(x)|^2 \big)^3 w_{B_{R_{k-1}}}(x)\,dx.
\end{equation}
By H\"{o}lder's inequality, 
\begin{equation}
\lesssim R^{O(\epsilon)}
    \sum_{B_{R_{k-1}}: B_{R_{k-1}} \cap U_{\alpha,\mathrm{mul}} \cap \Lambda_k \neq \emptyset }
    \int_{\mathbb{R}^3} \big( \sum_{\gamma_{k}}|f_{\gamma_{k}}(x)|^2 \big)^3 w_{B_{R_{k-1}}}(x)\,dx.
\end{equation}
The weight function $w_{B_{R_{k-1}}}$ decays fast outside of $B_{R_{k}}$. So we have
\begin{equation}\label{24.01.18.47}
\begin{split}
\lesssim R^{O(\epsilon)}
    \sum_{B_{R_{k-1}}: B_{R_{k-1}} \cap U_{\alpha,\mathrm{mul}} \cap \Lambda_k \neq \emptyset }
    &\int_{B_{R_{k}} } \big( \sum_{\gamma_{k}}|f_{\gamma_{k}}(x)|^2 \big)^3\,dx \\&+\mathrm{Rapdec}(R)(\sum_{\gamma}\|f_{\gamma}\|_{\infty}^2)^3.
\end{split}
\end{equation}
By doing some trivial estimate,
the first term of \eqref{24.01.18.47} is further bounded by
\begin{equation}\label{24.01.15.45}
\begin{split}
\lesssim R^{O(\epsilon)} 
    \sum_{B_{R_{k-1}}: B_{R_{k-1}} \cap U_{\alpha,\mathrm{mul}} \cap \Lambda_k \neq \emptyset } &\Big\| \sum_{\gamma_{k+1}}|f_{\gamma_{k+1}}|^2 \Big\|_{L^{\infty}(B_{R_{k+1}})}
    \\& \times
    \int_{B_{R_{k}} } \big( \sum_{\gamma_{k}}|f_{\gamma_{k}}(x)|^2 \big)^2\,dx.
\end{split}    
\end{equation}
Suppose that $B_{R_{k-1}} \cap U_{\alpha,\mathrm{mul}} \cap \Lambda_{k} \neq \emptyset$. Let $x_0$ be an element of the set. Then by Lemma \ref{24.01.15.lem33} and the definition of $\Lambda_{k}$, we have
\begin{equation}
    \Big\| \sum_{\gamma_{k+1}}|f_{\gamma_{k+1}}|^2 \Big\|_{L^{\infty}(B_{R_{k+1}})} \lesssim \Big\| \sum_{\gamma_{k+1}}|f_{\gamma_{k+1}}|^2 \Big\|_{L^{1}(w_{B_{R_{k+1}}}^{\#}(x_0))} \lesssim R^{O(\epsilon)}\sum_{\gamma \in \mPb}\|f_{\gamma}\|_{\infty}^2.
\end{equation}
By this inequality, \eqref{24.01.15.45} is bounded by
\begin{equation}\label{24.02.15.439}
\begin{split}
\lesssim R^{O(\epsilon)} \sum_{\gamma \in \mPb}\|f_{\gamma}\|_{\infty}^2
   \sum_{B_{R_{k-1}}: B_{R_{k-1}} \cap U_{\alpha,\mathrm{mul}} \cap \Lambda_k \neq \emptyset } 
    \int_{B_{R_{k}} } \big( \sum_{\gamma_{k}}|f_{\gamma_{k}}(x)|^2 \big)^2\,dx.
\end{split}    
\end{equation} 
By Lemma \ref{24.01.15.lem33},
\begin{equation}
\begin{split}
    \int_{B_{R_{k}} } \big( \sum_{\gamma_{k}}|f_{\gamma_{k}}(x)|^2 \big)^2\,dx & \lesssim
|B_{R_{k}}| \Big\|\sum_{\gamma_{k}}|f_{\gamma_{k}}|^2 \Big\|_{L^\infty(B_{R_{k}})}^2 \\& \lesssim |B_{R_{k}}|
\Big\| \sum_{\gamma_{k}}|f_{\gamma_{k}}|^2 \Big\|_{L^{1}(w_{B_{R_{k}}}^{\#}(x_0))}^2
\\& \lesssim \int_{B_{R_{k}}} |g_{k}(x)|^2\,dx,
\end{split}
\end{equation}
where $x_0$ is an element in the ball $B_{R_{k}}$. Hence, \eqref{24.02.15.439} is bounded by
\begin{equation}
    \lesssim R^{O(\epsilon)} \sum_{\gamma \in \mPb}\|f_{\gamma}\|_{\infty}^2
    \sum_{B_{R_{k}}: B_{R_{k}} \cap  \Lambda_{k} \neq \emptyset } 
    \int_{B_{R_{k}} } |g_{k}|^2.
\end{equation}
We now apply Corollary \ref{24.01.15.cor36}, and obtain
\begin{equation}
    \lesssim R^{O(\epsilon)} \sum_{\gamma \in \mPb}\|f_{\gamma}\|_{\infty}^2\int_{\R^3}\sum_{\gamma_{k}}|f_{\gamma_{k}}|^4.
\end{equation}
By Lemma \ref{24.02.13.lem42}, and Cauchy-Schwarz inequality, this is bounded by
\begin{equation}
\begin{split}
&\lesssim R^{O(\epsilon)} \sum_{\gamma \in \mPb}\|f_{\gamma}\|_{\infty}^2\int_{\R^3}\sum_{\gamma_{k}}\Big|\sum_{\gamma \subset \gamma_{k}}|f_{\gamma}|^2\Big|^2
    \\&\lesssim R^{O(\epsilon)}R^{\beta}R_{k}^{-1}\sum_{\gamma \in \mPb}\|f_{\gamma}\|_{\infty}^2\sum_{\gamma} \int |f_{\gamma}|^4.\end{split}    
\end{equation}
Here we used that the number of $\gamma$ in $\gamma_{k}$ is bounded by $O(R^{\beta}R_{k}^{-1})$.

To summarize, we have obtained
\begin{equation}
    \begin{split}\alpha^6|U_{\alpha,\mathrm{mul}} \cap \Lambda_{k}| &\lesssim R^{O(\epsilon)}R^{\beta}R_k^{-1}\sum_{\gamma \in \mPb}\|f_{\gamma}\|_{\infty}^2\sum_{\gamma \in \mPb } \int |f_{\gamma}|^4
    \\& \lesssim 
    R^{O(\epsilon)}R^{2\beta}R_k^{-1}\sup_{\gamma \in \mPb}\|f_{\gamma}\|_{\infty}^2\sum_{\gamma \in \mPb } \int |f_{\gamma}|^4
    \end{split}
\end{equation}
By Lemma \ref{24.01.18.lem411},
\begin{equation}
        \alpha \lesssim R^{O(\epsilon)}(R_{k})^{\frac12}(\sum_{\gamma}\|f_{\gamma}\|_{\infty}^2)^{\frac12} \lesssim R^{O(\epsilon)}(R_k)^{\frac12}R^{\frac{\beta}{2}}\sup_{\gamma}\|f_{\gamma}\|_{\infty}.
    \end{equation}
Combining these two, we have
\begin{equation*}
    \alpha^{6+2/\beta}|U_{\alpha,\mathrm{mul}} \cap \Lambda_{k}| \lesssim 
    (R_k)^{\frac{1}{\beta}}R
    R^{O(\epsilon)}R^{2\beta}R_k^{-1}
    \sup_{\gamma}\|f_{\gamma}\|_{\infty}^{2+\frac{2}{\beta}}
    \sum_{\gamma} \int |f_{\gamma}|^4.
\end{equation*}
To prove \eqref{24.01.19.42}, it suffices to check
\begin{equation}
    (R_k)^{\frac1\beta-1}R^{2\beta+1} \lesssim  R^{(3\beta+1)-\beta}.
\end{equation}
This is true by the assumption that $\beta \geq 1$.
\medskip

\subsection{Estimate for  \texorpdfstring{$|U_{\alpha,\mathrm{mul}} \cap \Lambda_{k}|$}{} for \texorpdfstring{$k=0$}{}}

By Lemma \ref{24.01.18.lem411},
\begin{equation}\label{24.01.19.432}
    \alpha \lesssim R^{O(\epsilon)} (\sum_{\gamma}\|f_{\gamma}\|_{\infty}^2)^{\frac12} \lesssim  R^{O(\epsilon)}R^{\beta/2}\sup_{\gamma}\|f_{\gamma}\|_{\infty}.
\end{equation}
On the other hand, Lemma \ref{24.02.13.lem42}, and H\"{o}lder's inequality,
\begin{equation}\label{24.01.19.433}
\begin{split}
    \alpha^4|U_{\alpha,\mathrm{mul}} \cap \Lambda_{0}| &\lesssim \int |f|^4
    \\&\lesssim \int \Big| \sum_{\gamma}|f_{\gamma}|^2 \Big|^2 \lesssim R^{\beta} \sum_{\gamma} \int |f_{\gamma}|^4.
\end{split}
\end{equation}
Combining
\eqref{24.01.19.432} and \eqref{24.01.19.433} gives
\begin{equation}
    \alpha^{6+\frac{2}{\beta}}|U_{\alpha,\mathrm{mul}} \cap \Lambda_{0}| \lesssim R^{O(\epsilon)} R^{\beta} R^{\frac{\beta}{2}(2+\frac{2}{\beta})}\big(\sup_{\gamma}\|f_{\gamma}\|_{\infty}\big)^{2+\frac{2}{\beta}} \sum_{\gamma}\int |f_{\gamma}|^4.
\end{equation}
This gives \eqref{24.01.19.42}.

\section{Proof of Proposition \ref{parabolathm} \label{parapf}}

Proposition \ref{parabolathm} follows from the small cap decoupling theorem (Theorem 3.1) of \cite{MR4153908}, which we describe using our notation presently. 

\begin{theorem}[Theorem 3.1 of \cite{MR4153908}] Let $\beta\in[\frac{1}{2},1]$ and consider $R^{-\beta}\times R^{-1}$ small caps $\gamma$. Let $2\le p$. Then 
\[   \|f\|_{L^p(B_R)}^p\lesssim_\e R^\e [R^{p\beta-2\beta-1}+R^{(\frac{p}{2}-1)\beta}] (\sum_\gamma\|f_\gamma\|_{L^p(\R^2)}^p)^{p/2} \]
for any Schwartz function $f:\R^2\to\mathbb{C}$ with Fourier transform supported in the $R^{-1}$-neighborhood of $\{(t,t^2):0\le t\le 1\}$. 
\end{theorem}
Theorem 3.1 of \cite{MR4153908} applies to more general convex curves, as does Proposition \ref{parabolathm}, but we focus on the parabola to simplify notation. 

\begin{proof}[Proof of Proposition \ref{parabolathm}]

In Proposition \ref{parabolathm}, perform the change of variables $x\mapsto(N^{-1}x_1,N^{-2}x_2)$, so that we need to prove that
\begin{align}\label{needtoshow}
\int_{[0,N]\times[0,N^{2-\sigma}]} |\sum_{k=1}^Na_ke(x&\cdot(\frac{k}{N},\frac{k^2}{N^2}))|^p dx \\
&\lesssim_\e N^\e (N^{\frac{\sigma}{2}(p-4)+3}+N^{p-1})\sum_{k=1}^N|a_k|^p. \nonumber
\end{align}  
Note that $N^{2-\sigma}<N$. Let $\psi:\R^2\to[0,1]$ be a smooth Schwartz function satisfying $\psi\sim 1$ on $[0,N]\times[0,N^{2-\sigma}]$ and having Fourier transform supported in $[0,N^{-1}]\times[0,N^{\sigma-2}]$. Let $\{\tau\}$ be a collection of $\sim N^{\sigma-2}$ intervals which partition $[0,1]$ and let $f_\tau(x)=\psi(x)\sum_{\frac{k}{N}\in\tau}a_k e(x\cdot(\frac{k}{N},\frac{k^2}{N^2}))$. We will show \eqref{needtoshow} in two cases: $2\le p\le 4$ and then $p\ge 4$. For $2\le p\le 4$, the small cap decoupling theorem for the parabola (Theorem 3.1 of \cite{MR4153908}) says that the left hand side of \eqref{needtoshow} is bounded by a $C_\e N^\e$ factor of
\[ (N^{2-\sigma})^{(\frac{p}{2}-1)}\sum_\tau\|f_\tau\|_{L^p(\R^2)}^p . \]
Analyze each $\|f_\tau\|_{L^p(\R^2)}^p$:
\begin{align*}
\|f_\tau\|_{L^p(\R^2)}^p&=\int_{\R^2}|\psi(x)\sum_{\frac{k}{N}\in\tau}a_k e(x\cdot(\frac{k}{N},\frac{k^2}{N^2}))|^pdx \\
    &\le (\sum_{\frac{k}{N}\in\tau}|a_k|)^{p-2}\int_{\R^2}|\psi(x)\sum_{\frac{k}{N}\in\tau}a_k e(x\cdot(\frac{k}{N},\frac{k^2}{N^2}))|^2dx.   
\end{align*}
The Fourier supports of the summands are disjoint, so the final line above is bounded using $L^2$ orthogonality by 
\[ (\sum_{\frac{k}{N}\in\tau}|a_k|)^{p-2}\int_{\R^2}\sum_{\frac{k}{N}\in\tau}|a_k|^2|\psi(x)|^2 dx\sim (\sum_{\frac{k}{N}\in\tau}|a_k|)^{p-2}(\sum_{\frac{k}{N}\in\tau}|a_k|^2)N^{3-\sigma}. \]
By H\"{o}lder's inequality, using that the number of $\frac{k}{N}\in\tau$ is $\lesssim N^{\sigma-1}$, the right hand side above is bounded by $(N^{\sigma-1})^{(p-2)}\sum_{\frac{k}{N}\in\tau}|a_k|^p N^{3-\sigma}$. Combining this bound for each $\|f_\tau\|_{L^p(\R^2)}^p$ with the result of small cap decoupling gives
\[ \text{L.H.S. of \eqref{needtoshow}}\lesssim_\e N^\e N^{\frac{\sigma}{2}p-2\sigma+3}\sum_k|a_k|^p,\]
as desired. It remains to check the $p\ge 4$ case. Here, we rely on the $p=4$ case that we just justified:
\begin{align*}
    \text{L.H.S. of \eqref{needtoshow}}&\le (\sum_k|a_k|)^{p-4} \int_{\R^2}|\psi(x)\sum_{k=1}^Na_ke(x\cdot(\frac{k}{N},\frac{k^2}{N^2}))|^4dx \\
    &\lesssim_\e N^\e (\sum_k|a_k|)^{p-4}.
\end{align*}
This completes the proof.
\end{proof}

\bibliographystyle{alpha}
\bibliography{reference}

\end{document}